\documentclass[12pt,leqno]{article}

\usepackage[lite]{amsrefs}
\usepackage{latexsym,enumerate}

\usepackage{amssymb, amsfonts, eucal,amsmath,amsthm}
\usepackage[cp1251]{inputenc}
\usepackage[OT2,T2A]{fontenc}
\usepackage[ukrainian,english]{babel}

\newtheorem{theorem}{Theorem}[section]

\newtheorem{lemma}[theorem]{Lemma}

\newtheorem{definition}[theorem]{Definition}

\newtheorem{corollary}[theorem]{Corollary}

\newcommand{\sect}[1]{\section{#1} \setcounter{equation}{0} }

\newcounter{ca}
\newcommand{\norm}[2]{\left\|#1\right\|_{#2}}

\newcommand{\ds}{\displaystyle}

 \newcommand{\ec}{\end{comment}}
\newcommand{\bc}{ \begin{comment}
 }

 \newcommand\uw{\mathrm{w}}
 \newcommand{\omx}{\lambda_{x}}

\newcommand{\LL}{{\mathcal L}}

\newcommand{\A}{{\mathcal A}}

\newcommand{\DD}{{\mathcal D}}
\newcommand{\Q}{{\mathcal Q}}

\newcommand{\andd}{\quad\mbox{\rm and}\quad}

\newcommand\w{{\omega}}

\def\be  {\begin{equation}}
\def\ee  {\end{equation}}
\def\ba  {\begin{eqnarray}}
\def\ea  {\end{eqnarray}}
\def\baa {\begin{eqnarray*}}
\def\eaa {\end{eqnarray*}}
\newenvironment{comment}[2]
{\bgroup\vspace{7pt}
\begin{tabular}{|p{5in}|}
\hline \qquad \bf \footnotesize Comment -- to be deleted in the final version \\
\hline
\quad\sl\footnotesize #1#2} {\\ \hline \end{tabular}
\vspace{7pt}\indent\egroup}

\def\updots{\mathinner{\mkern
1mu\raise 1pt \hbox{.}\mkern 2mu \mkern 2mu \raise
4pt\hbox{.}\mkern 1mu \raise 7pt\vbox {\kern 7 pt\hbox{.}}} }

\def \dist{\mathop{\rm dist}\nolimits}
\def \const{\mathop{\rm const}\nolimits}

\newcommand{\R}{\mathbb R}

\newcommand{\N}{\mathbb N}

\newcommand{\ineq}[1]{{\rm(\ref{#1})}}

\newcommand{\ie}{{\em i.e., }}
\newcommand{\eg}{{\em e.g. }}

\newcommand{\st}{\;\; \big| \;\;}

\newcommand{\thm}[1]{Theorem~\ref{#1}}
\newcommand{\lem}[1]{Lemma~\ref{#1}}
\newcommand{\cor}[1]{Corollary~\ref{#1}}

\title{{\sc On  one  estimate  of divided differences and its  applications 
}\thanks{{\it AMS classification:} 41A10, 41A25. {\it Keywords
and phrases:} divided difference, Whitney, Marchaud, Lagrange, Hermite, Dzyadyk,  interpolation, trace, extension, modulus of smoothness }}

\author{K. A. Kopotun\thanks
{Department of Mathematics, University of Manitoba,Winnipeg, Manitoba, R3T 2N2, Canada ({\tt
kopotunk@cc.umanitoba.ca}). Supported by NSERC of Canada.}\ \
D. Leviatan\thanks{Raymond and Beverly Sackler School of Mathematical
Sciences, Tel Aviv University, Tel Aviv 6139001, Israel ({\tt
leviatan@tauex.tau.ac.il}).}\ \
and \ I. A. Shevchuk\thanks
{Faculty of Mechanics and Mathematics, Taras
Shevchenko National University of Kyiv, 01601 Kyiv, Ukraine ({\tt
shevchuk@univ.kiev.ua}).}
}

\begin{document}

\maketitle

\abstract{We give an estimate of the general divided differences   $[x_0,\dots,x_m;f]$, where some of the $x_i$'s are allowed to coalesce (in which case, $f$ is assumed to be sufficiently smooth).  This estimate is then applied to significantly strengthen  Whitney  and Marchaud celebrated inequalities in relation to Hermite interpolation.  

For example, one of the numerous  corollaries of this estimate is the fact that, given a   function $f\in C^{(r)}(I)$ and a set $Z=\{z_j\}_{j=0}^\mu$ such that
$z_{j+1}-z_j \geq \lambda |I|$,  for all $0\le j \le \mu-1$, where $I:=[z_0, z_\mu]$, $|I|$ is the length of $I$ and $\lambda$ is some positive number,
 the Hermite polynomial $\LL(\cdot;f;Z)$ of degree $\le r\mu+\mu+r$ satisfying
$\LL^{(j)}(z_\nu; f;Z) = f^{(j)}(z_\nu)$, for all  $0\le \nu \le \mu$ and  $0\le j\le r$, approximates $f$ so that, for all $x\in I$,
\[
\big|f(x)-  \LL(x;f;Z) \big| \le C  \left( \dist(x, Z) \right)^{r+1} \int_{\dist(x, Z)}^{2|I|}\frac{\w_{m-r}(f^{(r)},t,I)}{t^2}dt ,
\]
where  $m :=(r+1)(\mu+1)$, $C=C(m, \lambda)$  and $\dist(x, Z) :=  \min_{0\le j \le \mu} |x-z_j|$.

\begin{center}
{\bf Абстракт}

\end{center}
Ми даємо оцінку узагальненої розділеної різниці $[x_0,\dots,x_m;f]$, де деякі з точок $x_i$ можуть співпадати  (в цьому випадку $f$ вважається досить гладкою). Ця оцінка потім застосовується для суттєвого посилення  відомих нерівностей Уітні  і Маршу та узагальнює їх для поліноміальної інтерполяції Ерміта.

Наприклад, одним з численних наслідків цієї оцінки є той факт, що для заданої функції $f\in C^{(r)}(I)$ та набору точок $Z=\{z_j\}_{j=0}^\mu$ таких, що
$z_{j+1}-z_j \geq \lambda |I|$,  для всіх $0\le j \le \mu-1$, де $I:=[z_0, z_\mu]$, $|I|$ є довжиною $I$ та  $\lambda$ є деяким додатнім числом,
поліном Ерміта
 $\LL(\cdot;f;Z)$ степеня $\le r\mu+\mu+r$, який задовольняє $\LL^{(j)}(z_\nu; f;Z) = f^{(j)}(z_\nu)$, для  $0\le \nu \le \mu$ та  $0\le j\le r$,
  наближує  $f$ так, що, для всіх $x\in I$,
\[
\big|f(x)-  \LL(x;f;Z) \big| \le C  \left( \dist(x, Z) \right)^{r+1} \int_{\dist(x, Z)}^{2|I|}\frac{\w_{m-r}(f^{(r)},t,I)}{t^2}dt ,
\]

де  $m :=(r+1)(\mu+1)$, $C=C(m, \lambda)$  та $\dist(x, Z) :=  \min_{0\le j \le \mu} |x-z_j|$.
}

\sect{Introduction}

V. K. Dzyadyk had a significant impact on
the theory of extension of functions, and we start this note with recalling three of his most significant results (in our opinion) in this direction.

First, in 1956 (see \cite{Dz56}), he solved a problem posed by S. M. Nikolskii on extending a   function $f\in{\rm Lip}_M(\alpha,p)$, $0<\alpha \le 1$, $p\ge 1$, on a finite interval $[a,b]$, to a function $F\in{\rm Lip}_{M_1}(\alpha,p)$ on the whole real line, \ie $F|_{[a,b]}=f$.

Then, in 1958 (see \cite{Dz58} or \cite{Dz}*{p. 171-172}),   he showed that if  $f\in C[0,1]$  then this function may be extended to a  function $F\in C[-1,1]$ with a controlled second modulus of smoothness on $[-1,1]$, \ie
$F|_{[0,1]}=f$, and the
second moduli of smoothness of $f$ and $F$ satisfy $\omega_2(F,\delta;[-1,1])\le 5\omega_2(f,\delta; [0,1])$, $0<\delta\le1$. (This result was  independently  proved by Frey \cite{F}  the same year.)

In this note, we   mostly deal with  results related to Dzyadyk's third result which we will now describe.

Given a function $f\in C[a,b]$ and $a\le x_0<x_1<x_2\le b $, the second divided difference $[x_0,x_1,x_2;f]$   can be estimated as follows (see, \eg\cite{Dz}*{p. 176} and \cite{DS}*{p. 237}):
\be\label{0}
|[x_0,x_1,x_2;f]|\le\frac c{x_2-x_0}\int_h^{x_2-x_0}\frac{\omega_2(f,t)}{t^2}dt,
\ee
where $c=\const<18$,  $h:=\min\{x_1-x_0, x_2-x_1\}$.

Now, let $\omega_2$ be an arbitrary function of the second modulus of smoothness type, \ie $\omega_2\in C[0,\infty]$ is   nondecreasing   and such that $\omega_2(0)=0$ and $t_1^{-2}\omega_2(t_1)\le4t_2^{-2}\omega_2(t_2)$,  $0<t_2<t_1$.

In 1983, Dzyadyk and Shevchuk \cite{DS83} proved that if $f$, defined on an arbitrary set $E\subset\R$,  satisfies \ineq{0} with $\omega_2(t)$ instead of $\omega_2(f,t)$ for each triple  of points $x_0, x_1, x_2 \in E$ satisfying $x_0<x_1<x_2$, then $f$ may be extended from $E$ to a function $F\in C(\R)$  such that  $\omega_2(F,t;\mathbb R)\le c\omega_2(t)$. In other words, \ineq{0} with $\omega_2(t)$ instead of $\omega_2(f,t)$ is necessary and sufficient for a function $f$ to be the trace, on the set $E\subset\R$, of a function $F\in C(\R)$ satisfying $\omega_2(F,t;\R)\le c\omega_2(t)$.
This result was independently proved by Brudnyi and Shvartsman \cite{BS} in 1982 (see also Jonsson \cite{J} for $\omega_2(t)=t$).

V. K. Dzyadyk posed the question to describe such traces for functions of the $k$th modulus of smoothness type with $k>2$. He conjectured  that an analog of  \ineq{0}  must be a corollary of
 Whitney and Marchaud inequalities. In 1984, this conjecture was confirmed by Shevchuk in  \cite{Sh84pre}, and a corresponding (exact) analog of \ineq{0}  for $k>2$ was found (see \ineq{mainin} below with $r=0$). Earlier, the case $\omega(t)=t^{k-1}$ was proved by Jonsson whose paper \cite{J} was submitted in 1981, revised in 1983 and published in 1985.

So what happens when we have differentiable functions?
In 1934, Whitney \cite{W}  described the traces of $r$ times continuously differentiable functions $F:\R\mapsto\R$  on arbitrary closed sets $E\subset \R$: this trace consists of   all functions  $f: E\mapsto\R$ whose $r$th   differences converge on $E$ (see \cite{W-diff} for the definition).
 In 1975, de Boor \cite{dB75}  described the traces of functions $F:\R\mapsto\R$ with bounded $r$-th derivative on   arbitrary  sets $E\subset \mathbb{R}$ of isolated points: this trace consists of all functions whose $r$-th divided differences are uniformly bounded on $E$ (in 1965, Subbotin \cite{Su} obtained   exact constants in the case when sets $E$ consist  of   equidistant points).

Finally, given an arbitrary set $E\subset\R$,   the  necessary and sufficient condition for a function $f$ to be a trace (on $E$)   of a function $F\in C^{(r)}(\R)$ with a prescribed $k$-th modulus of continuity of the $r$-th derivative was obtained by Shevchuk in 1984 in \cite{Sh84pre}; see also
\cite{S}*{Theorems 11.1 and 12.3}, \cite{DS}*{Theorems 3.2 and 4.3 in Chapter 4}  and \cite{SZ}, where a linear extension operator
was given.

In fact, this necessary and sufficient condition is an analog of \ineq{0} for  the $k$-th modulus of continuity of the $r$-th derivative of $f$ which is inequality \ineq{mainin}
 in   \thm{main} below. However, the original proof of \thm{main} was distributed among several publications (see \cites{Sh84pre, Sh84, Galan} as well as \cite{S} and \cite{DS}), and there was an unfortunate misprint in the formulation of  \cite{DS}*{Theorem 6.4 in Section 3}: in (3.6.36), ``$k$'' was written instead of ``$m$''.
 Hence, the main purpose of this note is to properly formulate this theorem (\thm{main}), provide its complete self-contained proof and discuss several important corollaries/applications that have been inadvertently overlooked  in the past.

\sect{Definitions, notations and the main result}

For $f\in C[a,b]$ and any $k\in\N$, set
\[
\Delta^k_u(f,x;[a,b]):=\begin{cases}
  \sum_{i=0}^k(-1)^i\binom ki f(x+(k/2-i)u),&\quad x\pm (k/2)u\in[a,b],\\
0,&\quad{\rm otherwise},
\end{cases}
\]
and denote by
\be \label{modulus}
\w_k(f,t;[a,b]):=\sup_{0<u\le t}\|\Delta^k_u(f,\cdot;[a,b])\|_{C[a,b]}
\ee
the $k$th modulus of smoothness of $f$ on $[a,b]$.

Now, we recall the definition of Lagrange-Hermite divided differences (see \eg \cite{DL}*{p. 118}). Let $X = \{x_j\}_{j=0}^m$ be a collection of $m+1$ points with possible repetitions. For each $j$, the multiplicity $m_j$ of $x_j$ is the number of $x_i$ such that $x_i=x_j$, and let $l_j$ be the number of $x_i=x_j$ with $i\le j$. We say that a point $x_j$ is a simple knot if its multiplicity is $1$. Suppose that a real valued function $f$ is defined at all points in $X$ and, moreover, for each $x_j\in X$, $f^{(l_j-1)}(x_j)$ is defined as well (\ie $f$ has $m_j-1$ derivatives at each point that has multiplicity $m_j$).

Denote
\[
[x_0;f]:=f(x_0),
\]
the divided difference of $f$ of order $0$ at the point $x_0$.

\begin{definition} Let $m\in\mathbb{N}$.
 If $x_0=\dots=x_m$, then we denote
\[
[x_0, \dots, x_m; f] = [\underbrace{x_0, \dots, x_0}_{m+1};f]:=\frac{f^{(m)}(x_0)}{m!}.
\]
Otherwise,   $x_0\ne x_{j^*}$, for some number $j^*$, and we denote
\[
[x_0,\dots,x_m;f]:=\frac1{x_{j^*}-x_0} \left([x_1,\dots,x_{m};f]-[x_0,\dots,x_{j^*-1},x_{j^*+1},\dots,x_m;f]\right),
\]
the divided (Lagrange-Hermite) difference of $f$ of order $m$ at the knots $X = \{x_j\}_{j=0}^m$.
\end{definition}

Note that $[x_0,\dots,x_m;f]$ is symmetric in $x_0, \dots, x_m$ (\ie it does not depend on how the points from $X$ are numbered), and recall that
\begin{align} \label{hermite}
L_m(x; f) & := L_m(x; f; x_0, \dots, x_m) \\ \nonumber
& := f(x_0)+  \sum_{j=1}^m [x_0, \dots, x_j; f](x-x_0)\dots (x-x_{j-1})
\end{align}
is the (Hermite) polynomial of degree $\le m$ that satisfies
\be \label{intcond}
L_m^{(l_j-1)}(x_j; f) = f^{(l_j-1)}(x_j) , \quad \text{for all }\;   0\le j\le m .
\ee

Hence, in particular, if $x_{j_*}$ is a simple knot, then we can write
\be \label{diffwh}
[x_0,\dots,x_m;f]:=\frac{f(x_{j_*})-L_{m-1}(x_{j_*};f;x_0,\dots,x_{j_*-1},x_{j_*+1},\dots,x_m)}{\prod_{j=0,j\ne j_*}^m(x_{j_*}-x_j)} .
\ee

From now on, for convenience, we assume that all interpolation points are numbered from left to right, \ie the set of interpolation points $X = \{x_j\}_{j=0}^m$ is such that  $x_0\le x_1 \le \dots \le x_m$. We also assume 
that
the maximum multiplicity of each point is $r+1$ with $r\in\N_0$,  so that
\be \label{multipl}
x_j < x_{j+r+1}, \quad \text{for all} \quad 0\le j \le m-r-1 .
\ee
Also, let
\begin{align} \label{setpq}
\Q_{m,r}  := & \left\{ (p,q) \st 0\le p, q \le m \andd q-p\ge r+1 \right\}   \\ \nonumber
  =& \left\{ (p,q) \st 0\le p \le m-r-1 \andd p+r+1\le q \le m \right\} ,
\end{align}
and note that $\Q_{m,r} = \emptyset$ if $m\le r$.

Now, for all $(p,q)\in\Q_{m,r}$, put
\[
d(p,q):=  d(p,q; X)    :=\min\{x_{q+1}-x_p,x_{q}-x_{p-1}\},
\]
where $x_{-1}:=x_0-(x_m-x_0)$ and $x_{m+1}:=x_m+(x_m-x_0)$. Note, in particular, that
\[
d:= d(X) :=  d(0,m; X)    =2(x_m-x_0).
\]

Everywhere below, $\Phi$ is the set of nondecreasing functions $\varphi \in C[0,\infty]$  satisfying $\varphi(0)=0$. We also denote
\[
\Lambda_{p,q,r}(x_0,\dots,x_m;\varphi):=\frac{\ds \int_{x_q-x_p}^{d(p,q)}u^{p+r-q-1}\varphi(u)du}{\ds \prod_{i=0}^{p-1}(x_q-x_i)\prod_{i=q+1}^{m}(x_i-x_p)}, \quad (p,q)\in\Q_{m,r},
\]
and
\[
\Lambda_{r}(x_0,\dots,x_m;\varphi):=\max_{(p,q)\in\Q_{m,r}}  \Lambda_{p,q,r}(x_0,\dots,x_m;\varphi).
\]
Here, we use the usual convention that $\prod_{i=0}^{-1}:=1 $ and  $\prod_{i=m+1}^{m} := 1$.

The following theorem is the main result of this paper.

\begin{theorem}\label{main}  Let $r\in\N_0$ and $m\in\N$ be such that $m\ge r+1$, and suppose that a set
$X = \{x_j\}_{j=0}^{m}$ is such that  $x_0\le x_1 \le \dots \le x_m$ and \ineq{multipl} is satisfied.
If  $f\in C^{(r)}[x_0,x_m]$, then
\be\label{mainin}
\left|[x_0,\dots,x_m;f] \right|\le c\Lambda_{r}(x_0,\dots,x_m;\omega_{k}),
\ee
where $k:= m-r$ and  $\omega_{k}(t):=\omega_{k}(f^{(r)},t ; [x_0,x_m])$, and the constant $c$ depends only on $m$.
\end{theorem}

\sect{Auxiliary lemmas}

Throughout this section, we assume that $r\in\N_0$, $m\in\N$, $m\ge r+1$, the set $X = \{x_j\}_{j=0}^{m}$ is such that  $x_0\le x_1 \le \dots \le x_m$ and \ineq{multipl} is satisfied, and that $(p,q)\in\Q_{m,r}$. For convenience, we also denote $k:= m-r$.

 We first show that \thm{main} is valid in the case $m=r+1$ (\ie $k=1$).

\begin{lemma}\label{k=1} Theorem \ref{main} holds if $m=r+1$.

\begin{proof}
If  $m=r+1$, then $\Q_{m,r} = \{(0,r+1)\}$,
  and so
\[
\Lambda_{r}(x_0,\dots,x_m;\varphi) = \Lambda_{0,r+1,r}(x_0,\dots,x_m;\varphi) = \int_{d/2}^{d} u^{-2} \varphi(u) du.
\]
Hence, since $x_0\ne x_m$ by assumption \ineq{multipl},   \ineq{mainin}   follows from the identity
\begin{align*}
[x_0,\dots,x_m;f]=\frac{[x_1,\dots,x_{r+1};f]-[x_0,\dots,x_{r};f]}{x_m-x_0}
=\frac{f^{(r)}(\theta_1)-f^{(r)}(\theta_2)}{r! d/2},
\end{align*}
where $\theta_1\in(x_1,x_{r+1})$ and $\theta_2\in(x_0,x_{r})$, and the estimate
\[
\frac{|f^{(r)}(\theta_1)-f^{(r)}(\theta_2)|}{d} \le   \frac{\omega_1(d/2)}d \le \int_{d/2}^d\frac{\omega_1(u)}{u^2}dt=\Lambda_{r}(x_0,\dots,x_m;\omega_1). \quad \qed
\]
\renewcommand{\qedsymbol}{}
\end{proof}
\end{lemma}

For $k>2$,  we need the following lemma.
\begin{lemma}\label{lemma3} Let $(p,q)\in\Q_{m,r}$ be such that $q-p+2\le m$. If $\varphi\in\Phi$ and $\omega\in\Phi$ are such that 
\be\label{l2}
\varphi(t)\le t^{k-1}\int_t^{d}u^{-k}\omega(u)du,\quad t\in(0,d/2],
\ee
then
\be\label{l3}
\Lambda_{p,q,r}(x_0,\dots,x_{m};\varphi)\le 2^{k^2}\Lambda_{r}(x_0,\dots,x_{m};\omega) .
\ee
\end{lemma}

\begin{proof}  
Let $(p,q)\in\Q_{m,r}$ such that $q-p+2\le m$ be fixed, and  consider the   collection $\{(p_\nu,q_\nu)\}_{\nu=0}^{m-q+p}$ which we define as follows.
Let $(p_0,q_0):=(p,q)$, and for   $\nu \ge 1$,
\[
(p_\nu,q_\nu):=
\begin{cases}
(p_{\nu-1}-1,q_{\nu-1}),\quad &\text{if}\quad  x_{q_{\nu-1}}-x_{p_{\nu-1}-1}\le x_{q_{\nu-1}+1}-x_{p_{\nu-1}},\\
(p_{\nu-1},q_{\nu-1}+1),\quad &\text{otherwise}.
\end{cases}
\]
It is clear that $q_\nu-p_\nu = q_{\nu-1}-p_{\nu-1}+1$, and so
\be \label{auxid}
q_\nu-p_\nu = q-p +\nu ,
\ee
and one can easily check (for example, by induction) that, for all $1\le \nu \le m-q+p$,
\[
0\le p_\nu\le p_{\nu-1}< q_{\nu-1}\le q_{\nu}\le m.
\]
Hence, in particular,
\[
(p_{m-q+p},q_{m-q+p})=(0,m).
\]

In the rest of this proof, we use the notation
\[
d_\nu := d(p_\nu, q_\nu),  \quad 0\le\nu \le m-q+p.
\]

Also, observe that
\[
d_\nu  \ge d_{\nu-1}=x_{q_\nu}-x_{p_\nu}\quad1\le\nu\le m+q-p,
\]
and
\[
d_{m-q+p-1}=x_m-x_0 = d/2.
\]

We now show that, for all $1\le \nu \le m-q+p$,
\be\label{11}
\frac{d_{\nu-1} }{\ds \prod_{i=0}^{p_{\nu-1}-1}(x_{q_{\nu-1}}-x_i)\prod_{i=q_{\nu-1}+1}^{m}(x_i-x_{p_{\nu-1}})}\le \frac{2^k}{\ds \prod_{i=0}^{p_{\nu}-1}(x_{q_{\nu}}-x_i)\prod_{i=q_{\nu}+1}^{m}(x_i-x_{p_{\nu}})}.
\ee
Indeed, if $x_{q_{\nu-1}}-x_{p_{\nu-1}-1}\le x_{q_{\nu-1}+1}-x_{p_{\nu-1}}$, then $(p_{\nu},q_{\nu})=(p_{\nu-1}-1,q_{\nu-1})$, $d_{\nu-1}=x_{q_{\nu-1}}-x_{p_{\nu-1}-1}$ and, for $q_{\nu-1}+1  \le j \le m$,
\begin{align*}
x_j-x_{p_{\nu}} & =(x_j-x_{q_{\nu-1}})+(x_{q_{\nu-1}}-x_{p_{\nu-1}-1})\le(x_j-x_{p_{\nu-1}})+(x_{q_{\nu-1}+1}-x_{p_{\nu-1}}) \\
& \le 2(x_j-x_{p_{\nu-1}}),
\end{align*}
whence
\[
\prod_{i=q_{\nu-1}+1}^{m}(x_i-x_{p_{\nu-1}})\ge 2^{q_{\nu-1}-m}\prod_{i=q_{\nu}+1}^{m}(x_i-x_{p_{\nu}}),
\]
that yields \ineq{11}  because $m-q_{\nu-1}\le m-q \le  k$.

Similarly, if $x_{q_{\nu-1}}-x_{p_{\nu-1}-1} >  x_{q_{\nu-1}+1}-x_{p_{\nu-1}}$, then $(p_{\nu},q_{\nu})=(p_{\nu-1},q_{\nu-1}+1)$, $d_{\nu-1}= x_{q_{\nu-1}+1}-x_{p_{\nu-1}}$, and, for $0\le j\le p_{\nu-1}-1$,
\begin{align*}
x_{q_{\nu}}-x_j & = (x_{q_{\nu-1}+1}-x_{p_{\nu-1}}) + (x_{p_{\nu-1}}-x_j) \\
& < (x_{q_{\nu-1}}-x_{p_{\nu-1}-1}) + (x_{q_{\nu-1}}-x_j) \le 2 (x_{q_{\nu-1}}-x_j) ,
\end{align*}
and whence
\[
\prod_{i=0}^{p_{\nu-1}-1}(x_{q_{\nu-1}}-x_i) \geq 2^{-p_{\nu-1}} \prod_{i=0}^{p_{\nu}-1}(x_{q_{\nu}}-x_i) ,
\]
that also yields \ineq{11} because $p_{\nu-1}\le p < k$.

Inequality \ineq{11}   implies that, for all $1\le \nu \le m-q+p$,
\be\label{12}
\frac{ \ds \prod_{i=0}^{\nu-1} d_i}{\ds \prod_{i=0}^{p-1}(x_q-x_i)\prod_{i=q+1}^{m}(x_i-x_p)}\le \frac{2^{k\nu}}{\ds \prod_{i=0}^{p_\nu-1}(x_{q_\nu}-x_i)\prod_{i=q_\nu+1}^{m}(x_i-x_{p_\nu})}.
\ee

It is clear that $d(p,q) \le   x_m-x_0 = d/2$, and so
condition  \ineq{l2} implies that
\begin{align*}
\int_{x_q-x_p}^{d(p,q)}{u^{p+r-q-1}}{\varphi(u)}du\le\int_{x_q-x_p}^{d(p,q)}{u^{p+m-q-2}}\left(\int_{u}^{d}v^{-k}\omega(v)dv\right)du .
\end{align*}
Using integration by parts
we write
\begin{align*}
(m-q+p-1)& \int_{x_q-x_p}^{d(p,q)}{u^{p+r-q-1}}{\varphi(u)}du-\int_{x_q-x_p}^{d(p,q)}{u^{p+r-q-1}}{\omega(u)}du\\
&\le d^{m-q+p-1}(p,q)\int_{d(p,q)}^d\frac{\omega(u)}{u^{k}}du  \\
& =d^{m-q+p-1}(p,q)\sum_{\nu=1}^{m-q+p}\int_{d_{\nu-1} }^{d_\nu}\frac{\omega(u)}{u^{k}}du\\
&\le 2 \sum_{\nu=1}^{m-q+p} \;  \prod_{i=0}^{\nu-1} d_i
\int_{d_{\nu-1} }^{d_\nu }u^{p+r-q-1-\nu}\omega(u)du.
\end{align*}
The last estimate  is obvious for $1\le \nu \le m-q+p-1$ and, for $\mu = m-q+p$, it follows from
\[
d_0^{m-q+p-1} d_{m-q-p} \leq 2 \prod_{i=0}^{m-q+p-1} d_i
\]
which is valid because
\[
d_0^{m-q+p-1} \le \prod_{i=0}^{m-q+p-2} d_i \andd d_{m-q-p} = d(0,m) = d   = 2 d_{m-q+p-1}.
\]
Finally, taking into account \ineq{auxid}, \ineq{12} and recalling that $d_{\nu-1} = x_{q_\nu}-x_{p_\nu}$,  $1\le \nu\le m-q+p$,
we obtain
\begin{align*}
(m   -q+p-1) & \Lambda_{p,q,r}(x_0,\dots,x_{m};\varphi) \\
& \le \Lambda_{p,q,r}(x_0,\dots,x_{m};\omega) + 2\sum_{\nu=1}^{m-q+p}2^{k\nu}\Lambda_{p_\nu,q_\nu,r}(x_0,\dots,x_{m};\omega) \\
\end{align*}
that implies \ineq{l3}.
\end{proof}

\begin{lemma}\label{lemma4}   If $k=m-r\ge 2$ and
 $\varphi\in\Phi$ and $\omega\in\Phi$ are such that 
\be\label{2}
\varphi(t)\le t^{k-1}\int_t^{d}u^{-k}\omega(u)du,\quad t\in(0,d/2],
\ee
and $\varphi(t)\le \omega(t)$, $t\in[d/2,d]$, then
\be\label{3}
\Lambda_{r}(x_0,\dots,x_{m-1};\varphi)\le c(x_m-x_0)\Lambda_{r}(x_0,\dots,x_{m};\omega)
\ee
and
\be\label{4}
\Lambda_{r}(x_1,\dots,x_{m};\varphi)\le c(x_m-x_0)\Lambda_{r}(x_0,\dots,x_{m};\omega) ,
\ee
where constants $c$ depend only on $k$.
\end{lemma}

\begin{proof}
We first note that \ineq{4} is a consequence of \ineq{3}. Indeed, given  $X=\{x_i\}_{i=0}^m$, define the set $Y=\{y_i\}_{i=0}^m$ by letting $y_i := - x_{m-i}$, $0\le  i\le m$.
Then, $y_0\le y_1 \le \dots \le y_m$, $y_m-y_0 = x_m-x_0$ (and so, in particular, $d(Y)=d(X)=d$),
\begin{align*}
d(p,q;Y) & = \min\{y_{q+1}-y_p,y_{q}-y_{p-1}\} = \min\{ x_{m-p}-x_{m-q-1}, x_{m-p+1}-x_{m-q} \} \\
& = d(m-q, m-p; X) = d(m-q, m-p)   ,
\end{align*}
 and it is not difficult to check that, for any $\psi\in\Phi$,
\begin{align*}
\Lambda_{p,q,r}(y_0,\dots,y_m;\psi)
= \Lambda_{m-q,m-p,r}(x_0,\dots,x_m;\psi)
\end{align*}
and 
\begin{align*}
\Lambda_{p,q,r}(y_0,\dots,y_{m-1};\psi)
= \Lambda_{m-q-1,m-p-1,r}(x_1,\dots,x_{m};\psi) .
\end{align*}
Hence, using the fact that $(p,q)\in \Q_{\mu,r}$ iff  $(\mu-q,\mu-p)\in \Q_{\mu,r}$, $\mu=m-1,m$   we have
\begin{align*}
\Lambda_{r}(x_0,\dots,x_{m};\omega) &= \max_{(p,q)\in\Q_{m,r}}  \Lambda_{p,q,r}(x_0,\dots,x_m;\omega) \\
& = \max_{(m-q,m-p)\in\Q_{m,r}}  \Lambda_{m-q,m-p,r}(y_0,\dots,y_m;\omega) \\
&= \Lambda_{r}(y_0,\dots,y_m;\omega)
\end{align*}
and
\begin{align*}
\Lambda_{r}(x_1,\dots,x_{m};\varphi) & =\max_{(p,q)\in\Q_{m-1,r}}  \Lambda_{p,q,r}(x_1,\dots,x_m;\varphi) \\
&= \max_{(m-q-1,m-p-1)\in\Q_{m-1,r}} \Lambda_{m-q-1,m-p-1,r}(y_0,\dots,y_{m-1};\varphi) \\
& = \Lambda_{r}(y_0,\dots,y_{m-1};\varphi),
\end{align*}
and so \ineq{4} follows from \ineq{3} applied to the set $Y$.

We are now ready to prove \ineq{3}.
Let $(p^*,q^*)\in\Q_{m-1,r}$ be such that
\[
\Lambda^*:=\Lambda_{p^*,q^*,r}(x_0,\dots,x_{m-1};\varphi)=\Lambda_{r}(x_0,\dots,x_{m-1};\varphi),
\]
and denote, for convenience, $X_m := \{x_0, \dots, x_m\}$ and $X_{m-1} := \{x_0, \dots, x_{m-1}\}$.

We consider four cases.

\medskip
{\bf Case I:}  $(p^*,q^*)=(0,m-1)$.

We put $h:=x_{m-1}-x_0$ and note that
$\ds\Lambda^*=\int_{h}^{2h}u^{-k}\varphi(u)du$.

If $h\le d/4$, then
\begin{align*}
2^{1-k}\Lambda^*& \le (2h)^{1-k}  \varphi(2h)              \le \int_{2h}^{d}u^{-k}\omega(u)du \le \int_{h}^{d/2}u^{-k}\omega(u)du+\int_{d/2}^{d}u^{-k}\omega(u)du\\
&\le
 \int_{h}^{d/2}u^{-k}\omega(u)du+d\int_{d/2}^{d}u^{-k-1}\omega(u)du\\
&=(x_m-x_0) \Big(\Lambda_{0,m-1,r}(x_0,\dots,x_{m};\omega)+2\Lambda_{0,m,r}(x_0,\dots,x_{m};\omega)\Big) \\
&\le 3(x_m-x_0) \Lambda_{r}(x_0,\dots,x_{m};\omega).
\end{align*}
If $h> d/4$, then
\begin{align*}
\Lambda^*& = \int_{h}^{d/2}{ u^{-k}}{\varphi(u)}du + \int_{d/2}^{2h}{ u^{-k}}{\varphi(u)}du\le\left(4/d\right)^{k-1}\varphi(d/2) +\int_{d/2}^{2h}{ u^{-k}}{\varphi(u)}du\\
&<4^k\int_{d/2}^{d}{ u^{-k}}{\varphi(u)}du\le4^k\int_{d/2}^{d}{ u^{-k}}{\omega(u)}du\le 4^kd\int_{d/2}^{d}{ u^{-k-1}}{\omega(u)}du\\
&=2\cdot4^k(x_m-x_0)\Lambda_{0,m,r}(x_0,\dots,x_{m};\omega) \le 2\cdot4^k(x_m-x_0)\Lambda_{r}(x_0,\dots,x_{m};\omega).
\end{align*}

\medskip
{\bf Case II:}   either (i) $q^*\ne m-1$, or (ii) $q^*=m-1$, $p^*>0$, and $x_m-x_{p^*}>x_{m-1}-x_{p^*-1}$

In this case,   $d(p^*, q^*; X_{m-1}) = d(p^*, q^*; X_{m})= x_{m-1}-x_{p^*-1}$, and so
\[
\Lambda^*=(x_m-x_{p^*})\Lambda_{p^*,q^*,r}(x_0,\dots,x_{m};\varphi)
\le(x_m-x_0)\Lambda_{p^*,q^*,r}(x_0,\dots,x_{m};\varphi).\\ 
\]
Since $q^*-p^*+2\le m$, we may apply \lem{lemma3} and obtain \ineq{3}.

\medskip
{\bf Case III:}   $q^*=m-1$, $p^*\ge 2$ and $x_m-x_{p^*}\le x_{m-1}-x_{p^*-1}$
\medskip

In this case,  $d(p^*, q^*; X_{m-1})  = x_{m-1}-x_{p^*-1}$ and $d(p^*, q^*; X_{m}) = x_m-x_{p^*}$.
Hence, taking into account that,
 for   $0\le i \le {p^*}-1$,
\[
x_m-x_i = x_m - x_{p^*} + x_{p^*}  -x_i \le x_{m-1 }- x_{{p^*}-1} + x_{p^*} - x_i
\le 2(x_{m-1}-x_i) ,
\]
we have
\begin{align*}
\Lambda_{{p^*},m-1,r}&(x_0,\dots,x_{m-1};\varphi)-(x_m-x_{p^*})\Lambda_{{p^*},m-1,r}(x_0,\dots,x_{m};\varphi)\\
&=
\prod_{i=0}^{{p^*}-1}(x_{m-1}-x_i)^{-1} \int_{x_{m}-x_{{p^*}}}^{x_{m-1}-x_{{p^*}-1}}{u^{{p^*}+r-m}}\varphi(u)du\\
&\le
2^{p^*}\prod_{i=0}^{{p^*}-1}(x_{m}-x_i)^{-1}
(x_{m}-x_{{p^*}-1})\int_{x_{m}-x_{{p^*}}}^{x_{m}-x_{{p^*}-1}}{u^{{p^*}+r-m-1}}\varphi(u)du \\
&=2^{p^*}(x_{m}-x_{{p^*}-1})\Lambda_{{p^*},m,r}(x_0,\dots,x_{m};\varphi).
\end{align*}
Since $m-p^*+2 \le m$, we may apply \lem{lemma3} to obtain \ineq{3}.

\medskip
{\bf Case IV:} $(p^*,q^*)=(1,m-1)$ and $x_m-x_1\le x_{m-1}-x_{0}$
\medskip

 In this case, we have
\begin{align*}
\Lambda^*=&\frac1{x_{m-1}-x_{0}}{\int_{x_{m-1}-x_1}^{x_{m-1}-x_0}u^{1-k}\varphi(u)du} \\
& \le\frac1{x_{m-1}-x_{0}}\int_{x_{m-1}-x_1}^{x_{m-1}-x_0}
\left(\int_u^dv^{-k}\omega(v)dv\right)du\\
\le&
\int_{x_{m-1}-x_0}^du^{-k}\omega(u)du + \frac1{x_{m-1}-x_{0}}\int_{x_{m-1}-x_1}^{x_{m-1}-x_0}u^{1-k}\omega(u)du  =: \A_1 + \A_2 .
\end{align*}
Now,
\begin{align*}
\A_1 
& \le
\int_{x_{m-1}-x_0}^{d/2}u^{-k}\omega(u)du    +   d \int_{d/2}^du^{-k-1}\omega(u)du \\
& =
(x_m-x_0) \Big( \Lambda_{0,m-1,r}(x_0,\dots,x_{m};\omega)  +   2\Lambda_{0,m,r}(x_0,\dots,x_{m};\omega) \Big) \\
&\le
3 (x_m-x_0) \Lambda_{r}(x_0,\dots,x_{m};\omega)
\end{align*}
and
\begin{align*}
\A_2 &=
 \frac1{x_{m-1}-x_{0}}\int_{x_{m-1}-x_1}^{x_{m}-x_1}u^{1-k}\omega(u)du
  + \frac1{x_{m-1}-x_{0}}\int_{x_{m}-x_1}^{x_{m-1}-x_0}u^{1-k}\omega(u)du\\
  & \le
  (x_m-x_1) \Lambda_{1, m-1, r}(x_0,\dots,x_{m};\omega)
  + \int_{x_{m}-x_1}^{x_{m-1}-x_0}u^{-k}\omega(u)du \\
  &\le
  (x_m-x_0) \Lambda_{1, m-1, r}(x_0,\dots,x_{m};\omega)
  + \int_{x_{m}-x_1}^{x_{m}-x_0}u^{-k}\omega(u)du \\
  & =
  (x_m-x_0) \Big( \Lambda_{1, m-1, r}(x_0,\dots,x_{m};\omega) + \Lambda_{1, m, r}(x_0,\dots,x_{m};\omega)  \Big) \\
  & \le 2 (x_m-x_0) \Lambda_{r}(x_0,\dots,x_{m};\omega).
\end{align*}

\end{proof}

\sect{Proof of Theorem \ref{main}}

\begin{proof}
We use induction on $k = m-r$. The base case $k=1$ is addressed in \lem{k=1}. Suppose now that $k\ge 2$ is given,
assume that   \thm{main} holds for   $k-1$ and prove it for   $k$.
 
Denote by $P_{k-1}$ the polynomial of best uniform approximation of $f^{(r)}$ on $[x_0,x_m]$ of degree at most $k-1$, and let $g$ be such that
\[
g^{(r)} :=f^{(r)} -P_{k-1}.
\]
Then
\[
\w_k(g^{(r)},t; [x_0,x_m]) = \w_k(f^{(r)},t; [x_0,x_m])   =: \w_k^f(t) ,
\]
and Whitney's inequality yields
\be \label{nr0}
\|g^{(r)}\|_{[x_0,x_m]}\le  c \w_k(f^{(r)} ,x_m-x_0; [x_0,x_m]) =  \w_k^f(x_m-x_0).
\ee
Hence,   the well known  Marchaud inequality:
\begin{quote}
 if $F\in C[a,b]$ and $1\le \ell < k$, then, for all
\[
\w_\ell(F, t; [a,b]) \le c(k) t^\ell \left( \int_t^{b-a} \frac{\w_k(F, u; [a,b])}{u^{\ell+1}} \, du +   \frac{ \norm{F}{[a,b]}}{(b-a)^\ell} \right) , \quad 0<t\le b-a ,
\]
\end{quote}
implies,
 for $0<t\le x_m-x_0$,
\begin{align} \label{nr}
\omega_{k-1}^g(t) &  :=\omega_{k-1}(g^{(r)},t; [x_0,x_m]) \\ \nonumber
& \le c t^{k-1} \left( \int_t^{x_m-x_0} \frac{\w_k^f(u)}{u^{k}} \, du +   \frac{ \w_k^f(x_m-x_0)}{(x_m-x_0)^{k-1}} \right) \\ \nonumber
& \le c t^{k-1}\int_t^{2(x_m-x_0)}\frac{\omega_{k}^f(u)}{u^k}du.
\end{align}

We also note that \ineq{nr0} implies, in particular, that for all $t\in [x_m-x_0, 2(x_m-x_0)]$,
\be \label{nr1}
\omega_{k-1}^g(t) \le  c \|g^{(r)}\|_{[x_0,x_m]}  \le c \w_k^f(x_m-x_0) \le c \w_k^f(t).
\ee

We now represent the divided difference in the form
\begin{align*}
(x_m-x_0)[x_0,\dots,x_m;f]&=(x_m-x_0)[x_0,\dots,x_m;g] \\
&=[x_1,\dots,x_{m};g] - [x_0,\dots,x_{m-1};g] \\
&=[y_0,\dots,y_{m-1};g] - [x_0,\dots,x_{m-1};g],
\end{align*}
where $y_j :=x_{j+1}$, $0\le j\le m-1$. By the induction hypothesis,
\[
|[x_0,\dots,x_{m-1};g]|\le c\Lambda_{r}(x_0,\dots,x_{m-1};\omega_{k-1}^g)
\]
and
\[
|[y_0,\dots,y_{m-1};g]|\le c\Lambda_{r}(y_0,\dots,y_{m-1};\omega_{k-1}^g).
\]

Now,  taking into account \ineq{nr}, \ineq{nr1}
and homogeneity of $\Lambda_{r}(z_0,\dots,z_{m};\psi)$ with respect to $\psi$,
\lem{lemma4} with $\varphi := \omega_{k-1}^g$ and $\omega := K\omega_{k}^f$, where $K$ is the maximum of constants $c$ in \ineq{nr} and \ineq{nr1},  implies that
\[
\Lambda_{r}(x_0,\dots,x_{m-1};\omega_{k-1}^g)\le c(x_m-x_0)\Lambda_{r}(x_0,\dots,x_{m};\omega_{k}^f)
\]
and
\begin{align*}
\Lambda_{r}(y_0,\dots,y_{m-1};\omega_{k-1}^g) & =  \Lambda_{r}(x_1,\dots,x_{m};\omega_{k-1}^g) \\
& \le c(x_m-x_0)\Lambda_{r}(x_0,\dots,x_{m};\omega_{k}^f),
\end{align*}
which yields \ineq{mainin}.
\end{proof}

\sect{Applications}

Throughout this section,    the set $X = \{x_j\}_{j=0}^{m-1}$ is assumed to be  such that  $x_0 \le  x_1 \le  \dots \le x_{m-1}$ (unless stated otherwise), and denote
$I:= [x_0, x_{m-1}]$ and $|I| = x_{m-1}-x_0$. Also, all constants written in the form    $C(\mu_1, \mu_2, \dots)$ may depend only on parameters $\mu_1$, $\mu_2$, ... and not on anything else.

We first recall that the classical Whitney interpolation inequality can be written in the following  form.

\begin{theorem}[Whitney inequality, \cite{W57}] \label{whth}
Let $r\in\N_0$ and $m\in\N$ be such that $m\ge \max\{r+1, 2\}$, and suppose that a set
$X = \{x_j\}_{j=0}^{m-1}$ is such that
\be \label{cond}
x_{j+1}-x_{j}\ge \lambda |I|,\quad \text{for all }\;  0\le j \le m-2 ,
\ee
where $0<\lambda \le 1$.
If $f\in C^{(r)}(I)$,  then
\[
\big| f(x)-L_{m-1}(x;f;x_0,\dots,x_{m-1}) \big| \le C(m,\lambda) |I|^r\omega_{m-r}(f^{(r)},|I|,I) , \quad x\in I,
\]
where $L_{m-1}(\cdot;f;x_0,\dots,x_{m-1})$ is the (Lagrange) polynomial of degree $\le m-1$ interpolating $f$ at the points in $X$.
\end{theorem}

We emphasize that condition \ineq{cond} implies that the points in the set $X$ in the above theorem are assumed to be sufficiently separated from one another.
A natural question is what happens if condition \ineq{cond} is not satisfied and, moreover, if some of the points in $X$ are allowed to coalesce.  In that case, $L_{m-1}(\cdot ;f;x_0,\dots,x_{m-1})$ is the Hermite polynomial whose derivatives interpolate corresponding derivatives of $f$ at points that have multiplicities more than $1$, and \thm{whth} provides no information on its error of approximation of $f$.

It turns out that one can use \thm{main} to provide an answer to this question and significantly strengthen \thm{whth}. As far as we know the formulation of the following theorem (which is itself a corollary of a more general \thm{maincor} below)  is new and has not appeared anywhere in the literature.

\begin{theorem} \label{auxcor}
Let $r\in\N_0$ and $m\in\N$ be such that $m\ge r+2$,
 and suppose that a set
$X = \{x_j\}_{j=0}^{m-1}$ is such that
\be \label{cond1}
x_{j+r+1}-x_{j}\ge \lambda |I|,\quad \text{for all }\;  0\le j \le m-r-2,
\ee
where $0<\lambda \le 1$.
If $f\in C^{(r)}(I)$,  then
\[
\big|f(x)-L_{m-1}(x;f;x_0,\dots,x_{m-1}) \big|\le C(m,\lambda) |I|^r \omega_{m-r} (f^{(r)},|I|,I),    \quad x\in I,
\]
where $L_{m-1}(\cdot ;f;x_0,\dots,x_{m-1})$ is the Hermite polynomial defined in \ineq{hermite} and \ineq{intcond}.
\end{theorem}

\thm{auxcor} is an immediate corollary of the following more general theorem. Before we state it, we need to introduce the following notation. Given
$X = \{x_j\}_{j=0}^{m-1}$ with  $x_0 \le  x_1 \le  \dots \le x_{m-1}$ and  $x\in [x_0, x_{m-1}]$, we renumber all points $x_j$'s so that their distance from $x$ is nondecreasing. In other words, let $\sigma =(\sigma_0, \dots, \sigma_{m-1})$ be a permutation of   $(0, \dots, m-1)$ such that
\be \label{perm}
|x-x_{\sigma_{\nu-1}}| \le |x-x_{\sigma_{\nu}}|, \quad \text{for all }\; 1\le \nu \le m -1.
\ee
Note that this permutation $\sigma$ depends on $x$ and is not   unique if there are at least two points from $X$ which are equidistant from  $x$. Denote also
\be \label{rdist}
\DD_{r} (x, X) := \prod_{\nu=0}^r |x-x_{\sigma_\nu}| , \quad 0\le r\le m-1.
\ee

\begin{theorem} \label{maincor}
Let $r\in\N_0$ and $m\in\N$ be such that $m\ge r+2$, and suppose that a set
$X = \{x_j\}_{j=0}^{m-1}$ is such that
\be \label{cond2}
x_{j+r+1}-x_{j}\ge \lambda |I|,\quad \text{for all }\;  0\le j \le m-r-2,
\ee
where $0<\lambda \le 1$.
If $f\in C^{(r)}(I)$,  then,  
for each $x\in I$,
\begin{align} \label{corconcl}
\big|f(x)- & L_{m-1}(x;f;x_0,\dots,x_{m-1}) \big| \\ \nonumber
& \le C(m,\lambda) \DD_{r} (x, X)   \int_{|x-x_{\sigma_{r}}|}^{2|I|}\frac{\w_{m-r}(f^{(r)},t,I)}{t^2}dt,
\end{align}
where $\DD_{r} (x, X)$ is defined in \ineq{rdist}, and
 $L_{m-1}(\cdot ;f;x_0,\dots,x_{m-1})$ is the Hermite polynomial defined in \ineq{hermite} and \ineq{intcond}.
\end{theorem}

Before proving \thm{maincor} we state another corollary.
First, if $k\in\N$ and $\uw(t):=\omega_{k}(f^{(r)},t ; I)$, then $t_2^{-k} \uw(t_2) \le 2^k t_1^{-k} \uw(t_1)$, for   $0<t_1 <t_2$.
Hence, denoting $\omx := |I| \sqrt[k]{ |x-x_{\sigma_{r}}|/|I|}$
and  noting that $|x-x_{\sigma_{r}}| \le \omx \le |I|$, we have, for $k\ge 2$,
\begin{align*}
\int_{|x-x_{\sigma_{r}}|}^{2|I|} \frac{\uw(t)}{t^2} \, dt & = \left(  \int_{|x-x_{\sigma_{r}}|}^{\omx} + \int_{\omx}^{2|I|} \right) \frac{\uw(t)}{t^2} \, dt \\
& \le
\uw(\omx) \int_{|x-x_{\sigma_{r}}|}^{\infty} t^{-2}\, dt  + 2^k \omx^{-k} \uw(\omx) \int_{0}^{2|I|} t^{k-2} \, dt \\
&=
\frac{\uw(\omx)}{|x-x_{\sigma_{r}}|} \left(1 + \frac{2^{2k-1}}{k-1} \right) .
\end{align*}

Therefore, we immediately get the following consequence of \thm{maincor}.

\begin{corollary} \label{cormaincor}
Let $r\in\N_0$ and $m\in\N$ be such that $m\ge r+2$, and suppose that a set
$X = \{x_j\}_{j=0}^{m-1}$ is such that condition \ineq{cond2} is satisfied.

If $f\in C^{(r)}(I)$,  then,  for each $x\in I$,
\begin{align} \label{corconclcor}
\big|f(x)-  L_{m-1}(x;f;x_0,\dots,x_{m-1}) \big|
& \le C(m,\lambda) \DD_{r-1} (x, X)  \w_{m-r}(f^{(r)},\omx,I) \\ \nonumber
&   \le C(m,\lambda) \DD_{r-1} (x, X)  \w_{m-r}(f^{(r)},|I|,I) ,
\end{align}
where $\omx := |I| \big( |x-x_{\sigma_{r}}|/|I|\big)^{1/(m-r)}$.
\end{corollary}

We are now ready to prove \thm{maincor}.

\begin{proof}[Proof of \thm{maincor}]
We   note that all constants $C$ below may depend only on $m$ and $\lambda$ and are different even if they appear in the same line. It is clear that we can assume that $x$ is different from all $x_j$'s.
So we let  $1\le i\le m-1$ and $x\in (x_{i-1}, x_i)$ be fixed, and  denote
\[
y_j:=
\begin{cases} x_j, \quad &\text{if}\quad 0\le j\le i-1,\\
x, \quad &\text{if}\quad j=i,\\
x_{j-1}, \quad &\text{if}\quad  i+1\le j\le m,
\end{cases}
\]
$Y := \{y_j\}_{j=0}^m$,   $d(Y) := 2(y_m-y_0) = 2(x_{m-1}-x_0) = 2|I|$,    $k:= m-r$, and $\w_k(t) := \w_{k}(f^{(r)},t,[y_0, y_m]) = \w_{k}(f^{(r)},t,I)$.
Condition \ineq{cond2}   implies that  $y_j < y_{j+r+1}$, for all $0\le j \le m-r-1$, and so
we can use \thm{main} to estimate $\big| [y_0, \dots, y_m; f] \big|$.
Now, identity \ineq{diffwh} with $j_* := i$ that yields $y_{j_*} = x$  implies
\begin{align} \label{auxes}
\big|f(x)- & L_{m-1}(x;f;x_0, \dots, x_{m-1}) \big|\\ \nonumber
 &=
\big|f(x)-L_{m-1}(x;f;y_0,\dots, y_{i-1}, y_{i+1}, \dots, y_{m }) \big| \\ \nonumber
 &= \big| [y_0, \dots, y_m; f] \big|  \prod_{j=0, j\ne i}^m |x-y_j| \\ \nonumber
 & \le c    \Lambda_r(y_0, \dots, y_m; \w_k) \prod_{j=0}^{m-1} |x-x_j| \\ \nonumber
 & \le c \DD_r(x, X) |I|^{k-1} \Lambda_r(y_0, \dots, y_m; \w_k) .
\end{align}
We also note that it is possible to show that
$\prod_{j=0}^{m-1} |x-x_j| \ge (\lambda/2)^{k-1} \DD_r(x, X) |I|^{k-1}$, and so the above estimate cannot be improved.

In order to estimate $\Lambda_r$, we  suppose  that $(p,q)\in\Q_{m,r}$ and estimate $\Lambda_{p,q,r}$. Since $q-p\ge r+1$, we have
\[
y_q-y_i \ge y_q-y_{p-1}  \ge y_{p+r+1}-y_{p-1} \ge \lambda |I| ,
\quad \text{for }\; 0\le i \le p-1 ,
\]
and
\[
y_i-y_p \ge y_{q+1}-y_p \ge y_{p+r+2}-y_p \ge \lambda |I|,
\quad \text{for }\; q+1\le i \le m.
\]
Hence,
\be \label{lam}
\Lambda_{p,q,r}(y_0, \dots, y_m; \w_k) \le C |I|^{q-m-p} \int_{y_q-y_p}^{2|I|} u ^{p+r-q-1} \w_k(u) du .
\ee

We  consider the following two cases.

\medskip
{\bf Case 1:} $q\ge p+r+2$, or $q= p+r+1$ and $x\not\in [y_p, y_q]$

It is clear that $y_q-y_p \ge \lambda |I|$, and so it follows from \ineq{lam} that
\[
\Lambda_{p,q,r}(y_0, \dots, y_m; \w_k) \le C |I|^{-k} \w_k(|I|) \le C |I|^{1-k} \int_{|I|}^{2|I|}   \frac{ \w_k(u)}{u^2}  du .
\]

{\bf Case 2:}  $q= p+r+1$ and $x\in [y_p, y_q]$

If $x=y_p$, then  $p=i$, $q=i+r+1$, and
$y_q-y_p = x_{i+r}-x \geq |x-x_{\sigma_r}|$.

If $x=y_q$, then  $q=i$, $p=i-r-1$, and
$y_q-y_p = x -x_{i-r-1}\geq |x-x_{\sigma_r}|$.

If $x\in (y_p, y_q)$, then  
$y_q-y_p = x_{p+r}-x_p$.
Since it is impossible that $|x-x_{\sigma_r}| > \max\{ x-x_p, x_{p+r}-x\}$, for this would imply that
$\{p, \dots, p+r\} \subset \{\sigma_0, \dots, \sigma_{r-1} \}$
 which cannot happen since these sets have cardinalities $r+1$ and $r$, respectively,
  we conclude that
 $|x-x_{\sigma_r}| \le  \max\{ x-x_p, x_{p+r}-x\} \le x_{p+r}-x_p$.  
Thus, in this case, \ineq{lam} implies that
\[
\Lambda_{p,q,r}(y_0, \dots, y_m; \w_k) \le C |I|^{1-k} \int_{|x-x_{\sigma_r}|}^{2|I|}   \frac{ \w_k(u)}{u^2}  du.
\]

Hence,
\[
\Lambda_r(y_0, \dots, y_m; \w_k) \le C |I|^{1-k} \int_{|x-x_{\sigma_r}|}^{2|I|}   \frac{ \w_k(u)}{u^2}  du ,
\]
which together with \ineq{auxes} implies \ineq{corconcl}.
\end{proof}

We state one more corollary to illustrate the power   of \thm{maincor}. Suppose that $Z  = \left\{ z_j\right\}_{j=0}^\mu$ with
  $z_0< z_1 < \dots < z_{\mu}$, and let $X = \{x_j\}_{j=0}^{m-1}$ with $m :=(r+1)(\mu+1)$ be such that  $x_{\nu(r+1)+j} := z_\nu$, for all $0\le \nu \le \mu$ and $0\le j\le r$. In other words,
\[
X = \left\{ \underbrace{z_0, \dots, z_0}_{r+1}, \underbrace{z_1, \dots, z_1}_{r+1}, \dots, \underbrace{z_{\mu}, \dots, z_{\mu}}_{r+1} \right\} .
\]
Now, given $f\in C^{(r)}[z_0, z_\mu]$,  let $\LL (x; f; Z) := L_{m-1}(x, f; x_0, \dots, x_{m-1})$ be the Hermite polynomial of degree $\le m-1 = r\mu+\mu+r$ satisfying
\be \label{her}
\LL^{(j)}(z_\nu; f; Z) = f^{(j)}(z_\nu), \quad \text{for all}\quad  0\le \nu \le \mu  \andd  0\le j\le r .
\ee
Also,
\[
\dist(x, Z) := \min_{0\le j \le \mu} |x-z_j| , \quad x\in \R.
\]

\begin{corollary} \label{onemorecor}
Let $r\in\N_0$ and $\mu\in\N$, and suppose that a set $Z  = \left\{ z_j\right\}_{j=0}^\mu$
 is such that
\[
z_{j+1}-z_j \geq \lambda |I|, \quad \text{for all }\; 0\le j \le \mu-1,
\]
where $0<\lambda \le 1$, $I  := [z_0,z_{\mu}]$ and $|I| := z_{\mu}-z_0$.
If $f\in C^{(r)}(I)$,  then, for each $x\in I$,
\begin{align*}
\big|f(x)-  \LL(x;f;Z) \big| & \le C  \left( \dist(x, Z) \right)^{r+1} \int_{\dist(x, Z)}^{2|I|}\frac{\w_{m-r}(f^{(r)},t,I)}{t^2}dt \\
& \le C  \left( \dist(x, Z) \right)^{r} \w_{m-r}\left(f^{(r)},|I| \left( \dist(x, Z)/|I|\right)^{1/(m-r)},I\right) \\
& \le C  \left( \dist(x, Z) \right)^{r} \w_{m-r}(f^{(r)},|I|,I) ,
\end{align*}
where  $m :=(r+1)(\mu+1)$,  $C=C(m, \lambda)$ and the  polynomial $\LL(\cdot;f;Z)$ of degree $\le m-1$ satisfies \ineq{her}.
\end{corollary}

As a final note, we remark that some of the  results that appeared in the literature follow from the results in this note. For example,
(i) the main theorem in \cite{Gopmz} immediately follows from \cor{onemorecor} with $\mu=1$, $z_0=-1$ and $z_1=1$,
(ii) \cor{cormaincor} is much stronger than the main theorem in \cite{Gop},
(iii) a particular case of \cite{K-sim}*{Lemmas 8 and 9} for $k=0$ follows from \cor{cormaincor},
(iv) several propositions in the unconstrained case in \cite{GLSW} follow from \cor{cormaincor},
(v) \cite{LP}*{Lemma 3.3, Corollaries 3.4-3.6} follow from \cor{cormaincor},
and
(vi) the proof of \cite{kls}*{Lemma 3.1} may be simplified if \cor{cormaincor} is used.

\end{document}